\documentclass{article}

\usepackage{amsfonts}
\usepackage{amssymb}
\usepackage{amsthm}
\usepackage{amsmath}
\usepackage{enumerate}
\usepackage{url}
\usepackage{hyperref}

\newtheorem{thm}{Theorem}[section]

\newtheorem{prop}[thm]{Proposition}
\newtheorem{cor}[thm]{Corollary}

\newtheorem{lem}[thm]{Lemma}

\newtheorem{defn}{Definition}

\renewcommand{\phi}{\varphi}
\newcommand{\M}{\ensuremath{\mathcal{M}}}
\renewcommand{\O}{\ensuremath{\mathcal{O}}}

\newcommand{\omegaoneck}{\ensuremath{\omega_1^{CK}}}
\newcommand{\lomegaone}{\ensuremath{\mathcal{L}_{\omega_1,\omega}}}
\newcommand{\bethalpha}{\beths{\alpha}}
\newcommand{\bethomegaoneck}{\beths{\omegaoneck}}

\newcommand{\beths}[1]{\ensuremath{\beth_{#1}}}
\newcommand{\lang}[1]{\ensuremath{\mathcal{L}_{#1}}}
\newcommand{\bk}{\ensuremath{\mathbf{K}}}
\newcommand{\bbk}{\ensuremath{\mathbb{K}}}
\newcommand{\bkb}{\ensuremath{\mathbf{K}_b}}

\begin{document}

\title{Hanf number for Scott sentences\\ of computable structures}

\author{S.\ S.\ Goncharov, J.\ F.\ Knight and I.\ Souldatos}
\maketitle

\begin{abstract}

The \emph{Hanf number} for a set $S$ of sentences in $\lomegaone$ (or some other 
logic) is the least infinite cardinal $\kappa$ such that for all $\varphi\in S$, if $\varphi$ has 
models in all infinite cardinalities less than $\kappa$, then it has models of all infinite 
cardinalities.  S-D.\ Friedman asked what is the Hanf number for Scott sentences of computable  
structures.  We show that the value is $\bethomegaoneck$.  The same argument proves
that $\bethomegaoneck$ is the Hanf number for Scott sentences of hyperarithmetical  
structures.
\end{abstract}

\section{Introduction}

Scott \cite{S} showed that for any countable structure $\mathcal{A}$ for a countable vocabulary, 
there is a sentence of $L_{\omega_1\omega}$ whose countable models are exactly the isomorphic copies 
of $\mathcal{A}$.  Such a sentence is called a \emph{Scott sentence} for $\mathcal{A}$.  
In this paper, we show that the Hanf number for Scott sentences of computable structures is 
$\beth_{\omega_1^{CK}}$, where $\omega_1^{CK}$ is the first non-computable ordinal.  
We say that $\tau$ is a \emph{computable vocabulary} if the set of symbols is 
computable, and there is a computable function giving the arities.  

\begin{defn}  

Let $\tau$ be a computable vocabulary, and let $\mathcal{A}$ be a $\tau$-structure with universe a subset of $\omega$.  
The structure $\mathcal{A}$ is \emph{computable} if its atomic diagram, $D(\mathcal{A})$, is computable.  We think of the elements as constants, and we identify sentences with their G\"{o}del numbers, so that $D(\mathcal{A})$ is a subset of $\omega$.    

\end{defn} 

The paper splits into two parts. In Section 2, we prove the following 
theorem, which establishes $\bethomegaoneck$ as an upper bound for the Hanf number for Scott 
sentences of computable structures.

\begin{thm}
\label{Main}  

Let $\mathcal{A}$ be a computable structure for a computable vocabulary $\tau$, and let $\phi$ be a Scott sentence for $\mathcal{A}$.  If $\phi$ has models of cardinality $\beth_\alpha$ for all $\alpha < \omega_1^{CK}$, then it has models of all infinite cardinalities.  

\end{thm} 

For an infinite cardinal $\kappa$ and an $L_{\omega_1\omega}$-sentence $\phi$, we say that $\phi$ \emph{characterizes} $\kappa$ if $\phi$ has a model of cardinality $\kappa$, but not in cardinality $\kappa^+$.  In Section 3, we exhibit specific examples of computable structures $\mathcal{A}_a$, corresponding to ordinal notations $a\in \O$, such that the Scott sentence of $\mathcal{A}_a$ characterizes $\beths{|a|}$, where $|a|$ is the ordinal with notation $a$.  This is Theorem \ref{ComputableBeths}.  Combining Theorems \ref{Main} and \ref{ComputableBeths}, we obtain the following.

\begin{thm} 

The Hanf number for Scott sentences of computable structures is equal to $\beths{\omegaoneck}$. 

\end{thm}

The Hanf number for Scott sentences of hyperarithmetical structures is also equal to 
$\beth_{\omega_1^{CK}}$. 
The proof that we give for Theorem \ref{Main} also shows that the Hanf number for Scott sentences 
of hyperarithmetical structures is at most $\beth_{\omega_1^{CK}}$, and the Scott sentences of 
computable structures witness that it is at least $\beth_{\omega_1^{CK}}$.  (Similar reasoning would 
show that for a countable admissible set $A$ with ordinal $\gamma$, the Hanf number for Scott 
sentences of structures in $A$ is~$\beth_{\gamma}$.  We will not discuss this.)  

\bigskip

In the remainder of the introduction, we give some conventions and basic definitions, and we recall some well-known results.  

\subsection{Background in infinitary logic}

The following two results are given in \cite{K}.  The first result, proved independently by Morley and by L\'{o}pez-Escobar, says that the Hanf number for $\lomegaone$ is $\beth_{\omega_1}$.  

\begin{thm} [Morley, L\'{o}pez-Escobar] 
\label{thm1}

Let $\Gamma$ be a countable set of sentences of $\lomegaone$.  If $\Gamma$ has models of cardinality $\beth_\alpha$ for all 
$\alpha<\omega_1$, then it has models of all infinite cardinalities.  

\end{thm}

The next result, proved independently by Morley and by Barwise, says that for a countable admissible set $A$ with ordinal $\gamma$, the Hanf number for the admissible fragment $\lang{A} = A\cap \lomegaone$ is at most $\beth_\gamma$.            

\begin{thm} [Morley, Barwise]
\label{thm2}

Let $A$ be a countable admissible set with $o(A) = \gamma$, and let $\phi$ be a sentence of $\lang{A}$.  If $\phi$ has models of cardinality $\beth_\alpha$ for all $\alpha <\gamma$, then it has models of all infinite cardinalities.    

\end{thm}  

The proofs of Theorems \ref{thm1} and \ref{thm2} use the Erd\"{o}s-Rado Theorem to produce a model of $\phi$ with an infinite indiscernible sequence, in a language with added Skolem functions.  The indiscernible sequence can be stretched to give models in arbitrarily large cardinalities.

\bigskip

We shall use ``computable'' infinitary formulas.  The \emph{computable infinitary formulas} are formulas of $\lomegaone$ in which the infinite disjunctions and conjunctions are over c.e.\ sets.  To make this precise, we would assign indices to the formulas, based on notations in Kleene's $\mathcal{O}$, as is done in \cite{AshKnight}.  The least admissible set that contains $\omega$ is 
$A = L_{\omega_1^{CK}}$.  The subsets of $\omega$ in $A$ are exactly the hyperarithmetical sets, and all computable (or hyperarithmetical) structures are elements of $A$.  The computable infinitary formulas (in a fixed computable vocabulary $\tau$) are essentially the same as the $\lomegaone$ formulas in the admissible fragment $\mathcal{L}_A$; that is, for any formula 
$\varphi(\bar{x})$ in $\mathcal{L}_A$, there is a computable infinitary formula $\psi(\bar{x})$ that is logically equivalent to 
$\varphi(\bar{x})$.   
 
For many computable structures $\mathcal{A}$, there is a computable infinitary Scott sentence $\varphi$.  By Theorem \ref{thm2}, if $\varphi$ has models in all infinite cardinalities less than $\beth_{\omega_1^{CK}}$, then it has models of all infinite cardinalities.  However, some computable structures do not have a computable infinitary Scott sentence.  In particular, this is so for the ``Harrison ordering'', a computable 
ordering of type $\omega_1^{CK}(1+\eta)$.  The computable infinitary sentences true in the Harrison ordering are exactly those true in the ordering of type $\omega_1^{CK}$.  In fact, for any countable admissible set $A$, with ordinal $\alpha$, there are structures in $A$ with no Scott sentence in the admissible fragment $\mathcal{L}_A$.  One such structure is an ordering of type $\alpha(1+\eta)$.    

We do not use the notion of Scott rank in this paper, so we shall not give a definition.  We mention, for general interest, a result of Nadel \cite{Nadel1}, \cite{Nadel2}, saying that for a computable, or hyperarithmetical, structure $\mathcal{A}$, there is a computable infinitary Scott sentence just in case the Scott rank is less than $\omega_1^{CK}$.  More generally, if the structure $\mathcal{A}$ is an element of a countable admissible set $A$ with ordinal $\gamma$, then it has a Scott sentence in $\mathcal{L}_A$ just in case the Scott rank is less than $\gamma$.  The result below follows from a general theorem of Ressayre \cite{R1},\cite{R2}.   

\begin{thm}\
\label{Saturation}

\begin{enumerate}

\item  If $\mathcal{A}$ and $\mathcal{B}$ are computable (or hyperarithmetical) structures satisfying the same computable infinitary sentences, then $\mathcal{A}\cong\mathcal{B}$.  

\item  If $\mathcal{A}$ is a computable (or hyperarithmetical) structure, and $\bar{a}$ and $\bar{b}$ are tuples satisfying the same computable infinitary formulas in $\mathcal{A}$, then there is an automorphism of $\mathcal{A}$ taking $\bar{a}$ to $\bar{b}$.  

\end{enumerate} 

\end{thm}

\subsection{Fra\"{i}ss\'{e} limits}

The computable structures that we produce in Section 3 will be ``Fra\"{i}ss\'{e} limits.''   In the discussion below, we will give slightly non-standard definitions.  We will state a simple result on existence of computable Fra\"{i}ss\'{e} limits that is not the most general, but is exactly suited to our needs.              

\begin{defn}

Let $\tau$ be a countable relational vocabulary.  Let $\bk$ be a set of $\tau$-structures, all finite.       

\begin{enumerate}

\item $\bk$ satisfies the \emph{hereditary property}, or $HP$, if for all $A\in \bk$, all proper substructures of $A$ are in $\bk$.     

\item  $\bk$ satisfies the \emph{joint embedding property}, or $JEP$, if for all $A,B\in \bk$, there exists $C\in \bk$ with embeddings $f:A\rightarrow C$ and $g:B\rightarrow C$.    

\item  $\bk$ satisfies the \emph{amalgamation property}, or $AP$, if for all $A,B,C\in \bk$ with embeddings $f:C\rightarrow A$ and $g:C\rightarrow B$, there is some $D\in\bk$, with embeddings $f':A\rightarrow D$ and $g':B\rightarrow D$, such that $f'\circ f = g'\circ g$.     

\item  $\bk$ is an \emph{age} if it satisfies $HP$, $JEP$, and $AP$.  

\end{enumerate}

\end{defn}

\noindent
\textbf{Remarks}.\

\begin{enumerate}

\item  For Fra\"{i}ss\'{e}, the vocabulary of an age may have function symbols, and the structures making up the age are finitely generated, but not necessarily finite.  For us, the vocabulary of an age will always be relational, and the structures in the age are finite.  

\item  Fra\"{i}ss\'{e}'s definition of \emph{age} omits the condition $AP$.  He proved results with and without this condition.  With 
$AP$, the limit structures are unique and homogeneous, as in the theorem below.  We added $AP$ to the definition above because we do not want to consider ages without $AP$, and we do not want to have to say everywhere ``age satisfying $AP$''. 

\end{enumerate} 

\begin{thm} [Fra\"{i}ss\'{e}]
\label{FTheorem}

Let $\bk$ be a countable age.  Then there is a countable structure $\mathcal{A}$, unique up to isomorphism, such that the isomorphism types of finite substructures of $\mathcal{A}$ are exactly the isomorphism types of structures in $\bk$.  Moreover, $\mathcal{A}$ is ``homogeneous'' in the sense that any isomorphism between finite substructures of $\mathcal{A}$ extends to an automorphism of $\mathcal{A}$.  

\end{thm}

For an account of the proof of Theorem \ref{FTheorem}, see the model theory textbook by Hodges \cite{H}.  It is not at all difficult.  We construct $\mathcal{A}$ as the union of a chain of finite structures $\mathcal{A}_s$, all isomorphic to elements of $\bk$.  We extend, step by step, with the goal of producing a structure that includes copies of all elements of $\bk$ as substructures and is homogeneous.  The Joint Embedding Property and the Amalgamation Property guarantee that there is always an appropriate next structure.  

\begin{defn}

For a countable age $\bk$, the structure $\mathcal{A}$ as in Theorem \ref{FTheorem} is called the \emph{Fra\"{i}ss\'{e} limit of $\bk$}.  

\end{defn}

We want Fra\"{i}ss\'{e} limits that are \emph{computable}.  The proof of Theorem \ref{FTheorem} is effective, given a nice computable list of the structures in the age, and an effective way to determine when one structure in this list embeds in another.  We give some definitions to make these things precise.  The first definition says what we mean by a nice computable list of structures in $\bk$.  In addition to saying how to compute the atomic diagram of each structure, the list gives the full universe, in terms of the standard list of finite sets $(D_n)_{n\in\omega}$.  

\begin{defn} [Computable representation]     

Let $\tau$ be a computable relational vocabulary, and let $\bk$ be an age consisting of $\tau$-structures.  A \emph{computable representation of $\bk$} is a computable sequence $\bbk$ such that

\begin{enumerate}

\item  for each $i$, $\bbk(i)$ is a pair $(e,n)$ such that $\varphi_e$ is the characteristic function of the atomic diagram of a structure in 
$\bk$, and $D_n$ is the universe of this structure,

\item  for each $\mathcal{C}\in \bk$, there is some $i$ with first component $e$ such that $\varphi_e$ is the atomic diagram of a copy of 
$\mathcal{C}$.  

\end{enumerate}

\end{defn}  

\noindent
\textbf{Note}.  Informally, we may identify a computable representation $\bbk$ of $\bk$ with the uniformly computable sequence of structures $(\mathcal{C}_i)_{i\in\omega}$ such that the first component of $\bbk(i)$ is a computable index for $\mathcal{C}_i$, but we bear in mind that the second component of $\bbk(i)$ is an index for the full universe of $\mathcal{C}_i$.  Knowing that the first component of 
$\bbk(i)$ is $e$, we can effectively determine whether a given $c$ is in the universe of $\mathcal{C}_i$, but given $e$, we 
cannot say that the universe has no more elements beyond those in a certain finite set.

\bigskip

The next definition says when one structure (on the list given by a computable representation $\bbk$) can be embedded into another.  

\begin{defn}

Let $\tau$ be a computable relational language, and let $\bk$ be an age consisting of $\tau$-structures.  Suppose that $(C_i)_{i\in\omega}$ is the sequence of structures given by a computable representation $\bbk$.  

\begin{enumerate}

\item  The corresponding \emph{embedding relation}, denoted by $E(\bbk)$, is the set of triples $(i,j,f)$ such that $f$ is an embedding of $C_i$ into $C_j$.

\item  We say that $\bbk$ has the \emph{strong embedding property} if $E(\bbk)$ is computable.

\end{enumerate}   

\end{defn}   

\noindent
\textbf{Remark}.  If $\tau$ is a finite relational vocabulary, then for any computable representation $\bbk$ of $\bk$, $E(\bbk)$ is computable.  If $\tau$ is infinite, this is not always true.      

\begin{prop}

There is a computable representation $\bbk$ of an age $\bk$ (for a computable vocabulary $\tau$) such that $E(\bbk)$ is not even c.e.

\end{prop}

\begin{proof} [Proof sketch]

Let $\tau$ consist of unary predicates $U_n$ for $n\in\omega$.  Let $\bk$ be the set of finite 
$\tau$-structures in which each element satisfies $U_n$ for at most one $n$.  The isomorphism type 
of a structure in $\bk$ is determined by the set of $n$ such that the structure has an element in 
$U_n$ and the number of elements not in any $U_n$.  We construct a computable representation $\bbk$ 
of $\bk$ such that $E(\bbk)$ is not c.e.  We describe the construction of a uniformly computable 
sequence $(C_i)_{i\in\omega}$ of $\tau$-structures, with universe specified.  The effective 
construction proceeds in stages.  At stage $s$, we determine, for each of finitely many $i$, the 
full universe of $C_i$ and a finite part of the atomic diagram.  The isomorphism types of the 
$C_i$'s must be exactly those of the structures in $\bk$, and we must satisfy the following 
requirements. 

\bigskip
\noindent
$R_e$:  $W_e$ is not equal to $E(\bbk)$.        
                
\bigskip

The strategy for $R_e$ is as follows.  At stage $s$, when we first begin work on the requirement, we designate a pair of indices $i$, $i+1$, on which we have not yet specified the universe or said anything about the atomic diagrams.  We give $C_i$ universe $2$ and $C_{i+1}$ universe $3$.  Let $f$ be the identity function on $2$.  We vow to put $1$ into $U_i$ in both structures, and to put $2$ into $U_{i+1}$ in $C_{i+1}$.  We keep $0$ out of $U_n$ in $C_{i+1}$.  We vow to keep $0$ out of all $U_n$ in $C_i$ unless the triple $(i,i+1,f)$ appears in $W_e$.  If this happens at stage $s$, then for the first $n$ such that we have not already put into the diagram of $C_i$ the statement $\neg{U_n(0)}$, we add the statement $U_n(0)$. 

\bigskip

We continue enumerating the diagrams of of structures $C_i$, making sure that the isomorphism types match those in $\bk$, and satisfying the requirements.      
So, by definition, $\bbk$ is a computable representation of $\bk$ and $E(\bbk)\neq W_e$, for all $e$, which proves the result.
\end{proof}           

Here are the last definitions we shall need in discussing computable Fra\"{i}ss\'{e} limits.  

\begin{defn}

Let $\bk$ be an age, with computable representation $\bbk$.  Let $(C_i)_{i\in\omega}$ be the corresponding sequence of structures.  
Let $\mathcal{A}$ be a Fra\"{i}ss\'{e} limit of $\bk$.    

\begin{enumerate}

\item  $E(\bbk,\mathcal{A})$ is the set of pairs $(i,f)$ such that $f$ is an embedding of $C_i$ into~$\mathcal{A}$.

\item  $\mathcal{A}$ is \emph{effectively homogeneous} if the set of finite partial isomorphisms between substructures of $\mathcal{A}$ is computable.  

\end{enumerate} 

\end{defn}

Here is the result that we will use in Section 3.  

\begin{thm} 
\label{ComputableFraisseLimit}

Let $\tau$ be a computable \emph{relational} language, possibly infinite.  Let $\bk$ be an age consisting of $\tau$-structures.  Suppose that $\bbk$ is a computable representation of $\bk$ with the strong embedding property.  Then there is a computable Fra\"{i}ss\'{e} limit 
$\mathcal{A}$ such that $E(\bbk,\mathcal{A})$ is computable.  In fact, we have a uniform effective procedure for passing from $\tau$, 
$\bbk$ and $E(\bbk)$ to $D(\mathcal{A})$ and $E(\bbk,\mathcal{A})$.        

\end{thm}

\begin{proof} [Proof Sketch]

The assumptions that $\bbk$ is a computable representation of $\bk$ and that $E(\bbk)$ is computable let us carry out the construction from \cite{H} effectively.  Say that $(C_i)_{i\in\omega}$ is the sequence of structures given by $\bbk$.  We construct the computable 
Fra\"{i}ss\'{e} limit $\mathcal{A}$ as the union of a uniformly computable sequence of finite structures $\mathcal{A}_s$, specifying at each step an isomorphism $f_s$ from some $C_i$ onto $\mathcal{A}_s$.  We determine a computable sequence of pairs 
$(i_s,f_s)_{s\in\omega}$ such that $f_s$ is an isomorphism from $C_{i_s}$ onto $\mathcal{A}_s$.  We know what to put into the diagram of $\mathcal{A}_s$ by looking $f_s$ and the diagram of $C_{i_s}$.    

To see that $E(\bbk,\mathcal{A})$ is computable, consider $f$ mapping the universe of $\mathcal{C}_i$ into $\mathcal{A}$.  For some $s$, we have $ran(f)\subseteq\mathcal{A}_s$, and we have specified a function $f_s$ mapping some $C_j$ isomorphically onto 
$\mathcal{A}_s$.  Let $g = f_s^{-1}\circ f$.  Then $(i,f)\in E(\bbk,\mathcal{A})$ iff $(i,j,g)\in E(\bbk)$.  

We have described a uniform procedure that takes the inputs $\tau$, $\bbk$ and $E(\bbk)$, and effectively produces $D(\mathcal{A})$ and 
$E(\bbk,\mathcal{A})$.           
\end{proof}  

We defined effective homogeneity.  The next result connects it with the relation $E(\bbk,\mathcal{A})$.   

\begin{prop}  

Suppose $\bk$ is an age with a computable representation $\bbk$ and a $\mathcal{A}$ is a computable Fra\"{i}ss\"{e} limit such that $E(\bbk,\mathcal{A})$ is computable.  Then $\mathcal{A}$ is effectively homogeneous. 

\end{prop} 

\begin{proof} 

We suppose that $\mathcal{A}$ has universe $\omega$.  Let $f$ be a finite partial $1-1$ function.  Find $i$ and $g$ such that $(i,g)\in E(\bbk,\mathcal{A})$, and let $h = f\circ g$.  Now, $f$ is an isomorphism between finite substructures of $\mathcal{A}$ iff $(i,h)\in E(\bbk,\mathcal{A})$.
\end{proof}   
 
In \cite{CHMM}, Csima et al give necessary and sufficient conditions for an age to give rise to a computable limit structure.  They allow function symbols in the vocabulary, and the structures in the age are finitely generated, but not necessarily finite.  Even assuming that the vocabulary is relational, the result in \cite{CHMM} does not match Theorem \ref{ComputableFraisseLimit}.  The hypotheses of Csima et al are weaker, and the conclusion is also weaker.  In particular, the embedding relation is not computable.  The result in \cite{CHMM} was inspired by an old result of Goncharov \cite{G} and Peretyat'kin \cite{P}, giving necessary and sufficient conditions for a countable homogeneous structure to have a decidable copy.  The proof in \cite{CHMM}, like those in \cite{G} and \cite{P}, involves a priority construction, with guesses at the extension relation, and injury resulting from guesses that are not correct.  This precludes effective homogeneity.  Theorem \ref{ComputableFraisseLimit} is much more elementary.       

In Section 3, we will construct, by induction, a family of computable limit structures $\mathcal{A}_\alpha$ corresponding to computable ordinals $\alpha$ (really, we will work with notations for ordinals).  For each $\alpha$, we obtain $\mathcal{A}_\alpha$ by applying Theorem \ref{ComputableFraisseLimit} to a triple of inputs 
$\tau_\alpha$, $\bbk_\alpha$, and $E(\bbk_\alpha)$.  it is straightforward to show that, given the inputs for $\mathcal{A}_\beta$ for $\beta < \alpha$, we can pass effectively to the inputs for $\mathcal{A}_\alpha$.  We first attempted this construction using the result in \cite{CHMM}, where the inputs for $\mathcal{A}_\alpha$ included only a weak substitute for $E(\bbk_\alpha)$.  Passing effectively from the inputs for $\mathcal{A}_\beta$ for $\beta < \alpha$ to the inputs for 
$\mathcal{A}_\alpha$ seemed too cumbersome.  We were pleased to find that we could apply the more elementary Theorem \ref{ComputableFraisseLimit}.     

\section{The Hanf number is at most $\beth_{\omega_1^{CK}}$}     

Our goal in this section is to prove that the Hanf number for Scott sentences of computable structures is at most $\beth_{\omega_1^{CK}}$.  The lemma below says that for a computable structure $\mathcal{A}$, we can replace the Scott sentence, which may not be computable infinitary, by a low level computable infinitary sentence in a larger vocabulary.  Let $\tau$ be a computable vocabulary, and let $\mathcal{A}$ be a computable 
$\tau$-structure.  From the original proof of the Scott Isomorphism Theorem \cite{S}, there is a family of $\lomegaone(\tau)$-formulas 
$\varphi_{\bar{a}}(\bar{x})$, corresponding to tuples $\bar{a}$ in $\mathcal{A}$, such that $\varphi_{\bar{a}}(\bar{x})$ defines the orbit of $\bar{a}$ under automorphisms of $\mathcal{A}$.  By Theorem \ref{Saturation} (b), we may take $\varphi_{\bar{a}}(\bar{x})$ to be the conjunction of the computable infinitary formulas true of~$\bar{a}$.             

\begin{lem}
\label{lem1}

Let $\tau$ be a computable vocabulary, and let $\mathcal{A}$ be a computable $\tau$-structure with Scott sentence $\phi$.  There is a computable vocabulary $\tau^*\supseteq \tau$ with a c.e.\ set $T$ of  computable infinitary $\tau^*$-sentences (all computable $\Pi_2$) such that for any $\tau$-structure $\mathcal{B}$, $\mathcal{B}\models \phi$ iff $\mathcal{B}$ has an expansion $\mathcal{B}^*$ satisfying $T$.  

\end{lem}

\begin{proof}

The vocabulary $\tau^*$ has predicates $P_{\bar{a}}$ for all tuples $\bar{a}\in\mathcal{A}$.  We put into $T$ sentences saying the following.

\begin{enumerate}

\item  $(\forall \bar{x})[P_{\bar{a}}(\bar{x})\rightarrow\varphi(\bar{x})]$, where $\varphi(\bar{x})$ is a finitary quantifier-free formula true of  
$\bar{a}$ in $\mathcal{A}$ (this is computable $\Pi_1$),

\item  $(\forall y) \bigvee_b P_b(y)\ \&\ \bigwedge_b(\exists y) P_b (y)$, where the disjunction and conjunction are over $b$ in $\mathcal{A}$ (this is computable $\Pi_2$),

\item  $(\forall\bar{x})[P_{\bar{a}}(\bar{x})\rightarrow ((\forall y) \bigvee_b P_{\bar{a},b}(\bar{x},y)\ \&\ \bigwedge_b (\exists y)P_{\bar{a},b}(\bar{x},y))]$, where $\bar{a}$ is a tuple in $\mathcal{A}$.  As for (2), the disjunction and conjunction are over $b$ in $\mathcal{A}$  (this is computable $\Pi_2$).

\end{enumerate}
Since $\mathcal{A}$ is computable, it is clear that $T$ is a c.e.\ set of computable $\tau^*$-sentences, all computable $\Pi_2$ or simpler.  We show that a $\tau$-structure $\mathcal{B}$ is a model of the Scott sentence $\phi$ iff it can be expanded to a model of $T$.  

\bigskip    
\noindent
($\Rightarrow$):  Suppose $\mathcal{B}$ is a model of the Scott sentence $\phi$.  We show that $\mathcal{B}$ can be expanded to a model 
$\mathcal{B}^*$ of $T$.  For $\bar{c}$ in $\mathcal{B}$, we put $\bar{c}$ into $P_{\bar{a}}^{\mathcal{B}^*}$ iff $\bar{c}$ satisfies in $\mathcal{B}$ the computable infinitary $\tau$-formulas that were true of $\bar{a}$ in $\mathcal{A}$.  There may be many tuples $\bar{a}'$ in $\mathcal{A}$ satisfying the same computable infinitary $\tau$-formulas as $\bar{a}$, and $\bar{c}$ will be in all of the corresponding relations $P_{\bar{a}'}^{\mathcal{B}^*}$.  We check that $\mathcal{B}^*$ is a model of $T$.  The sentences of type (1) are clearly true.  All of the relations 
$P_b$ are satisfied in $\mathcal{B}^*$, and each element of $\mathcal{B}^*$ satisfies at least one $P_b$.  Therefore, the sentences of type (2) are true.  Supposing that $\bar{c}$ satisfies $P_{\bar{a}}(\bar{u})$ in $\mathcal{B}^*$, there are elements $d$ satisfying $P_{\bar{a},b}(\bar{c},x)$, and every element $d$ satisfies one of these $P_{\bar{a},b}(\bar{c},x)$.  Therefore, the sentences of type (3) are true.      

\bigskip
\noindent
($\Leftarrow$):  Now, suppose that $\mathcal{B}$ has an expansion $\mathcal{B}^*$ satisfying $T$.  We must show that $\mathcal{B}$ satisfies $\phi$.  It is convenient to suppose that $\mathcal{B}^*$ is countable.  (In case it is not, we take the countable fragment $F$ generated by 
$\phi$ and the sentences of $T$.  We replace $\mathcal{B}^*$ by a countable $F$-elementary substructure $\mathcal{C}^*$, and we replace 
$\mathcal{B}$ by the appropriate reduct $\mathcal{C}$.)  Supposing that $\mathcal{B}^*$ is countable, we show that $\mathcal{B}$ satisfies 
$\phi$ by showing that $\mathcal{A}\cong\mathcal{B}$.  Let $\mathcal{F}$ be the set of finite partial functions mapping a non-empty tuple 
$\bar{a}$ in $\mathcal{A}$ to a tuple $\bar{b}$ in $\mathcal{B}$ such that $\mathcal{B}^*\models P_{\bar{a}}(\bar{b})$.  We show that 
$\mathcal{F}$ has the back-and-forth property.  Suppose $f\in\mathcal{F}$ maps $\bar{a}$ to $\bar{b}$.  For any $c$ in $\mathcal{A}$, there is some $d$ in $\mathcal{B}$ such that $\mathcal{B}^*\models P_{\bar{a},c}(\bar{b},d)$, so $f\cup\{(c,d)\}\in\mathcal{F}$.  For any $d$ in 
$\mathcal{B}$, there is some $c$ such that $\mathcal{B}^*\models P_{\bar{a},c}(\bar{b},d)$.  Then $f\cup\{(c,d)\}\in\mathcal{F}$. 

We note that the given $\mathcal{A}$ has a computable expansion to a model of $T$ in which, for each $\bar{a}$, the only tuple in the interpretation of $P_{\bar{a}}$ is $\bar{a}$ itself.  There is another expansion of $\mathcal{A}$ to a model of $T$, in which a tuple $\bar{c}$ is in the interpretation of $P_{\bar{a}}$ just in case $\bar{c}$ satisfies all of the computable infinitary $\tau$-formulas true of $\bar{a}$.  We do not claim that this second expansion is computable, but of course this does not matter. 
\end{proof}  

If $A$ is a countable admissible set containing the signature $\tau$ and the $\tau$-structure $\mathcal{A}$, then the set $T$, formed exactly as above, is c.e.\ relative to $\mathcal{A}$, and it consists of very simple sentences in an expanded signature $\tau^*$, where both $\tau^*$ and 
$T$ are in $A$.  Again, $\mathcal{B}$ is a model of the Scott sentence for $\mathcal{A}$ iff it can be expanded to a model of $T$.  

\bigskip      

Using Lemma \ref{lem1}, we can prove Theorem \ref{Main}.

\begin{proof} [Proof of Theorem \ref{Main}]

From the original Scott sentence $\phi$, in a computable vocabulary $\tau$, we pass to the c.e.\ set of sentences $T$ in the expanded vocabulary $\tau^*$, where $\tau^*$ is still computable.  Let $\phi^*$ be the conjunction of $T$.  This is a computable infinitary $\tau^*$-sentence.  For each $\alpha < \omega_1^{CK}$, the sentence $\phi$ has a model $\mathcal{B}$ of cardinality $\beth_\alpha$.  By Lemma \ref{lem1}, $\mathcal{B}$ can be expanded to a model $\mathcal{B}^*$ of $T$, and $\phi^*$. Applying Theorem \ref{thm2} to the computable infinitary $\tau^*$-sentence $\phi^*$, we get the fact that there are arbitrarily large models.  By Lemma~\ref{lem1}, the $\tau$-reducts of these all satisfy $\phi$.            
\end{proof} 

In the same way, we see that the Hanf number for Scott sentences of hyperarithmetical structures is at most $\beth_{\omega_1^{CK}}$.  In fact, for a countable admissible set $A$ with ordinal $\gamma$, the Hanf number for Scott sentences of structures in $A$ is at most $\beth_\gamma$.       

\section{The Hanf number is at least $\beth_{\omega_1^{CK}}$}\label{examples}

Recall that an infinite cardinal $\kappa$ is \emph{characterized} by an $\lomegaone$ sentence $\phi$ if $\phi$ has a model of cardinality $\kappa$ but does not have a model of cardinality $\kappa^+$.  For each $\alpha < \omegaoneck$, we construct a computable structure whose Scott sentence characterizes $\bethalpha$, thus proving that the Hanf number for Scott sentences of computable structures is \emph{exactly} equal to $\bethomegaoneck$.  In fact, we prove the following.  

\begin{thm}
\label{ComputableBeths} 

There exists a partial computable function $I$ such that for each $a\in \O$, $I(a)$ is a tuple of computable indices for several objects, among which are a relational vocabulary $\tau_a$, and the atomic diagram of a $\tau_a$-structure $\mathcal{A}_a$, with the following features: 

\begin{enumerate}  

\item  the Scott sentence $\phi_a$ of the structure $\mathcal{A}_a$ characterizes the cardinal $\beths{|a|}$, where $|a|$ is the ordinal with notation $a$,  

\item  the vocabulary $\tau_a$ contains a unary predicate $U_a$ and a binary relation $<_a$ such that

\begin{enumerate}
 
\item $(U_a,<)$ is a dense linear order without endpoints,

\item  there is a model $\mathcal{B}$ of $\phi_a$ of cardinality $\beths{|a|}$ such that 
$(U_a^{\mathcal{B}},<_a^{\mathcal{B}})$ has a co-final sequence of order type $\beths{|a|}$. 

\end{enumerate}

\end{enumerate}

\end{thm}  

We define $I$ by computable transfinite recursion on ordinal notation.  For each $a$, $I(a)$ is a tuple of computable indices for the following:

\begin{enumerate}

\item  the vocabulary $\tau_a$, 

\item  a representation $\bbk_a$ of an age $\bk_a$, 

\item  $E(\bbk_a)$,   

\item  the atomic diagram of $\mathcal{A}_a$, the Fraisse limit of $\bk_a$,  

\item  $E(\bbk,\mathcal{A}_a)$.

\end{enumerate}

The structure $\mathcal{A}_a$ along with the relation 
$E(\bbk,\mathcal{A}_a)$ are obtained by applying the uniform effective procedure of Theorem 
\ref{ComputableFraisseLimit} to $\tau_a$, $\bbk_a$ and $E(\bbk_a)$.  We must arrange that the Scott 
sentence 
$\phi_a$ for $\mathcal{A}_a$ characterizes the cardinal~$\beth_{|a|}$.    

\bigskip
\noindent
\textbf{Base case}.  Recall that $1$ is the unique notation for $0$.  We describe $I(1)$.  The vocabulary $\tau_1$ consists of unary relation symbols $U_1$ and $Q_q$ for $q\in\mathbb{Q}$, plus the binary relation symbol $<_1$.  We want $\mathcal{A}_1$ to be an expansion of 
$(\mathbb{Q},<)$ in which the interpretation of $U_1$ consists of all rationals, and the interpretation of $Q_q$ consists just of $q$.  The Scott sentence of $\mathcal{A}_1$ has no uncountable model.  The age $\bk_1$ consists of finite substructures of $\mathcal{A}_1$, including the empty structure.  It is not difficult to see that there is a computable representation $\bbk_1$ of $\bk_1$ for which the embedding relation 
$E_1 = E(\bbk_1)$ is computable.  We apply the uniform effective procedure from Theorem 
\ref{ComputableFraisseLimit} to get a computable limit structure 
$\mathcal{A}_1$ such that $E(\bbk_1,\mathcal{A}_1)$ is also computable.        

\bigskip    
\noindent
\textbf{Inductive step}.  We define $I(a)$, assuming that we have previously defined $I(b)$ for all $b <_\mathcal{O} a$, and $a\not= 1$.  Recall that for $a\in\mathcal{O}$, $|a|$ is the ordinal with notation $a$.  We split the construction into two cases, depending on whether $|a|$ is a successor ordinal or a limit ordinal.   

\subsection{Successor Ordinals}
\label{sec:Successor}

In this subsection, we suppose that $I$ has been defined on all $b \leq_\mathcal{O} a$ so that the conditions of Theorem \ref{ComputableBeths} are satisfied.  We suppose that $I(a)$ is a code for a quintuple of indices for $\tau_a$, $\bbk_a$, $E(\bbk_a)$, $D(\mathcal{A}_a)$, and 
$E(\bbk_a,\mathcal{A}_a)$.  The structure $\mathcal{A}_a$ is the Fra\"{i}ss\'{e} limit, which is 
obtained from $\bbk_a$ and $E(\bbk_a)$ as in Theorem \ref{ComputableFraisseLimit}, and the Scott 
sentence $\phi_a$ of $\mathcal{A}$ characterizes the cardinal $\beth_{|a|}$.  

By the induction hypothesis, we have a unary predicate $U_a$ and a binary relation $<_a$ such that 

\begin{enumerate}

\item [(a)] $(U_a,<)$ is a dense linear order without endpoints (in any model of $\phi_a$),  

\item [(b)] there is a model $\mathcal{B}$ of $\phi_a$ of size $\beths{|a|}$ such that $(U_a^{\mathcal{B}},<_a^{\mathcal{B}})$ contains a co-final sequence of order type $\beths{|a|}$. 

\end{enumerate}

Then we inductively extend the definition of $I$ to $b = 2^a$, where $|b| = |a|+1$.  The construction is a modified version of that in \cite{So2}.  We let $\tau_b$ be the vocabulary $\tau_a\cup \{V,M,U_b,P,F,<_b\}$, where $V$, $M$, and $U_b$ are unary predicates, $<_b$ is a binary predicate and $F$ is a ternary predicate.  We suppose that the symbols $V$, $M$, $U_b$, $P$, $F$, and $<_b$ are new, not in $\tau_a$.  We first describe $\bk_b$ and show that it is an age.  Then we consider the computable indices that make up $I(b)$.   

\bigskip

We let $\bk_b$ be the collection of all finite $\tau_b$-structures that satisfy the conjunction of the following: 

\begin{enumerate}

\item The domain is the disjoint union of $V$, $M$, $U_b$. Think of $V$ as a set of vertices and $M$ as a set of edge-colors and 
$U_b$ as a set of vertex-colors. 

\item $M\restriction\tau_a$ is a structure in $\bk_a$. In particular, there is a linear order $<_a$ defined on a subset $U_a$ of $M$. 

\item All relations in $\tau_a$ are void outside of $M$. 

\item  The predicate $P$ defines a vertex-coloring on $V$ with values in $U_b$.  That is, for each 
$v\in V$, there is \emph{at most} one $u\in U_b$ such that $P(v,u)$. 

\item The predicate $F$ defines an edge-coloring on $[V]^2\setminus\{(v,v)|v\in V\}$.  This time, the colors are elements of $U_a$; i.e., for each pair $v_0,v_1\in V$, there is \emph{at most} one $u\in U_a$ such that $F(v_0,v_1,u)$ and $F(v_1,v_0,u)$.  We will just write 
$F(v_0,v_1) = u = F(v_1,v_0)$. 

\item  $<_b$ is a linear order on $U_b$. 

The next property is the one that drives the construction. 

\item For any triple of distinct elements $v_0,v_1,v_2\in V$, if $F(v_0,v_1)\neq F(v_0,v_2)$, then 
\begin{align*}\label{star}\tag{$\bigstar$}
F(v_1,v_2) & =\min \{F(v_0,v_1),F(v_0,v_2)\},
\end{align*}
where $\min$ is according to the $<_a$-ordering.\\
Otherwise, $F(v_1,v_2)>_a F(v_0,v_1)=F(v_0,v_2)$. 
\end{enumerate}

\bigskip
\noindent
\textbf{Remark}.  The collection $\bk_b$ described above differs from the collection $K(\M)$ in \cite{So2} in the following respects: 

\begin{enumerate}

\item  The set $U_b$ and the projection $P$ are missing in $K(\M)$.  The reason it is introduced here is that we need it to carry out the induction. 

\item  Here the set $M$ is finite and its restriction to $\tau_a$ is a (finite) structure in $\bk_a$.  In \cite{So2}, the set $\M$ is infinite, and its restriction on some vocabulary $\tau$ satisfies an $\lomegaone(\tau)$-sentence $\phi$.  

\item  The requirement that $P$ and $F$ are total functions defined on their corresponding domains has been relaxed to solely requiring that they take at most one value. The reason is that we need $\bk$ to satisfy $HP$. This is not the case in \cite{So2}. Nevertheless, in the Fra\"{i}ss\'e limit, both $P$ and $F$ will be total functions, not just partial functions.

\item  The empty structure belongs to $\bk_b$, since it also belongs to $\bk_a$.

\end{enumerate}

\begin{lem}\label{lem:dap} 

$\bkb$ satisfies $HP$, $JEP$ and $AP$. 

\end{lem}

\begin{proof} 

The hereditary property follows immediately from the definition of $\bk_b$.  We will sketch the proof just for 
$AP$.  We get $JEP$ for free, since $\emptyset\in \bk_b$.  Let $A,B,C\in \bk_b$ where $C$ is a substructure of $A$ and $B$.  We need an amalgam $D\in \bk_b$ with embeddings $f:A\rightarrow D$ and $g:B\rightarrow D$ such that $f$ and $g$ agree on $C$.  Since $\bk_a$ satisfies $AP$, we can ate (the reducts to $\tau_a$ of $M^A$ and $M^B$ over $M^C$.  Let $D_1$ be the $\tau_b$-structure with $M^{D_1}$ equal to the resulting amalgam, and with $V^{D_1}$ and $U_b^{D_1}$ empty.  We may suppose that $D_1$ extends $M^C$, and that it is disjoint from $V^A\cup U_b^A$ and $V^B\cup U_b^B$.  Let $f_1$ embed $M^A$ into $M^D$, and let $g_1$ embed $M^B$ into $M^D$, where $f_1$ and $g_1$ agree with the the identity function on $M^C$. 

Next, using the argument from Lemma 4.9 of \cite{So2}, we amalgamate $V^A\cup M^{D_1}$ and 
$V^B\cup M^{D_1}$ over $V^C\cup M^{D_1}$, considering these as $\tau_b$-structures, with the appropriate interpretations of $F$.  Let $D_2$ be the amalgam structure, with $U_b^{D_2}$ empty.  We may suppose that $D_2$ extends $D_1$ and that it is disjoint from $U_b^A$ and $U_b^B$.  Let $f_2$ embed $D_1\cup V^A$ into $D_2$ and let $g_2$ embed $D_1\cup V^B$ into $D_2$, where $f_2$ and $g_2$ agree with the identity function on $D_1\cup V^C$.  The argument from Lemma 4.9 of \cite{So2} shows that $V^{D_2} = V^A\cup V^B$.  In the amalgam $D_2$, although no new points are added in forming $V^{D_2}$, it is possible that some new points are added in forming $M^{D_2}$ (that is, $M^{D_2}$ may have elements not in $M^{D_1}$).  

Finally, we amalgamate $D_2\cup U_b^A$ and $D_2\cup U_b^B$ over $D_2\cup U_b^C$, considering these as $\tau_b$ structures, with the appropriate interpretations of $P$.  The amalgam structure is the desired $D$.  We may suppose that $D$ extends $D_2\cup U_b^C$.  We let $f_3$ embed $D_2\cup U_b^A$ into $D$, and we let $g_3$ embed $D_2\cup U_b^B$ into $D$, where $f_3$ and $g_3$ agree with the identity on $D_2\cup U_b^C$.  Then $D$ is an extension of $C$.  Let $f$ be the restriction of $f_3$ to $A$, and let $g$ be the restriction of $g_3$ to $B$.  Then $f$ and $g$ are embeddings that agree on $C$.       
\end{proof}

Assuming that the symbols in $\tau_a$ are each marked by some notation $a'\leq_\mathcal{O} a$, we mark the finitely many new symbols by $b$.  We pass from an index for $\tau_a$ to an index for $\tau_b$.  We can easily pass from a computable representation $\bbk_a$ of $\bk_a$ to a computable representation $\bbk_b$ of 
$\bk_b$.  Say that $(C_i)_{i\in\omega}$ is the sequence of structures given by $\bbk_a$.  For $\bbk_b$, we will have sequence of structures $(D_i)_{i\in\omega}$ such that for each $i = \langle i_1,i_2\rangle$, 
$M^{D_i} = C_{i_1}$.  The other parts of $D_i$, namely, $V^{D_i}$ and $U^{D_i}$, are finite, with the relations 
$F$, $P$, and $<_b$ to be determined.  There are only finitely many symbols whose interpretation in $D_i$ is not determined by~$C_{i_1}$.  

We can compute $E(\bbk_b)$, using $E(\bbk_a)$ and $\bbk_b$.  To determine whether a finite function 
$f$ is an embedding of $D_i$ into $D_j$, we first check whether the appropriate restriction of $f$ 
embeds $C_{i_1}$ into $C_{j_1}$, and we then check that the finitely many further relations, 
involving new elements, are preserved.  From $\bbk_b$ and $E(\bbk_b)$, we compute $D(\mathcal{A}_b)$ 
and $E(\bbk_b,\mathcal{A}_b)$, as in Theorem \ref{ComputableFraisseLimit}.  Thus, we have $I(b)$, 
with 
computable indices for $\tau_b$, $\bbk_b$, $E(\bbk_b)$, $D(\mathcal{A}_b)$, and 
$E(\bbk_b,\mathcal{A}_b)$.  

\begin{thm}
\label{thm:Successor} 

There is a computable Fra\"{i}ss\'{e} limit $\mathcal{A}_b$ of $\bk_b$, with Scott sentence $\phi_b$, such that 

\begin{enumerate}
 
 \item $M^{\mathcal{A}_b}\restriction \tau_a$ is isomorphic to the $\tau_a$-structure $\mathcal{A}_a$,

\item $\phi_b$ characterizes the cardinal $\beths{|b|}$,
 
 \item $(U_b,<_b)$ is a dense linear order without endpoints, and 
 
 \item $\phi_b$ has a model $\mathcal{B}$ of size $\beths{|b|}$ such that 
 $(U_b^{\mathcal{B}},<_b^{\mathcal{B}})$ contains a co-final sequence of order type $\beths{|b|}$. 

\end{enumerate}

\end{thm}

\begin{proof} 

Clause (1) follows from the inductive hypothesis. Clause (2) follows from the proofs of Theorems 4.13 and 4.14 in \cite{So2}\footnote{See also Remark 4.28 and Theorem 4.29 in \cite{So2}.}, modified to include $U_b$ and $P$. As in the proof of Lemma \ref{lem:dap}, $U_b^D$ and $P^D$ are defined in the amalgam to be the union of $U_b^B\cup U_b^C$ and $P^B\cup P^C$, respectively.  The proof goes through because of disjoint amalgamation. 

Clause (3) follows from the usual proof that the Fra\"{i}ss\'{e} limit of finite linear orders yields a dense linear order without endpoints. Clause (4) follows from the fact that we can organize a sequence of $\beths{|b|}$-many amalgamation triples \linebreak
$(A_i,B_i,C_i)_{|i<\beths{|b|}}$, where the structures $A_i$, $B_i$, and $C_i$ are linearly ordered, $A_i$ and $B_i$ are finite, $C_i$ may be infinite, and $C_{i+1}$ is the amalgam of $B_i,C_i$ over $A_i$, so that $C_*=\bigcup_{i<\beths{|b|}} C_i$ contains a co-final sequence of order type $\beths{|b|}$. 
 \end{proof}

\subsection{Limit Ordinals}
\label{sec:Limit}

Assume $|a|$ is a limit ordinal.  Then $a$ has the form $3\cdot 5^e$, where $\varphi_e$ is a total recursive function with values 
$\varphi_e(n) = a_n$ such that $a_n <_\O a_{n+1}$ for all $n$, and $|a| = \lim_{n} |a_n|$. Without loss of generality, we may suppose that $|a_n|$ is a successor ordinal. 

Let $\tau_a$ be the union of the $\tau_{a_n}$'s, together with the new unary predicates $U_a$ and $Q_n$, for $n\in\omega$, and the new binary predicate $<_a$. Let $\bk_a$ consist of all $\tau_a$-structures (with universe a finite subset of $\omega$) such that 

\begin{enumerate}
 
\item $Q_n\restriction\tau_{a_n}$ is a structure in $\bk_{a_n}$,

\item the relations in $\tau_{a_n}$ are void outside $Q_n$,

\item $U_a = \bigcup_n U_{a_n}$, where each $U_{a_n}$ is the subset of $Q_n$ linearly 
ordered by $<_{a_n}$, and 

\item for $u_0,u_1\in U_a$, we define $u_0<_a u_1$ iff either there exists $n$ such that $Q_n(u_0)$ and 
$Q_n(u_1)$ and $u_0<_{a_n} u_1$, or there exist $n,m$ with $n < m$ such that $Q_n(u_0)$ and $Q_m(u_1)$.

\end{enumerate}  

Clearly, $\bk_a$ satisfies $HP$, $JEP$, and $AP$.  Given indices for $\bbk_{a_n}$ and $E(\bbk_{a_n})$, for all $n$, we can produce a computable representation $\bbk_a$ of $\bk_a$ such that $E(\bbk_a)$ is also computable.  We partition $\omega$ into disjoint sets $Q_n$.  For each $n$, let $p_n$ be the function mapping the elements of $\omega$ $1-1$ onto the elements of $Q_n$, in order.  We have a computable list of all finite partial functions $(\sigma_i)_{i\in\omega}$ from $\omega$ to $\omega$.  Let $(C^{a_n}_i)_{i\in\omega}$ be the sequence of structures given by $\bbk_{a_n}$, and let $(C^a_j)_{j\in\omega}$ be the sequence of structures given by $\bbk_a$.  We use $\sigma_i$ to determine $C^a_i$.  For each $n\in dom(\sigma_i)$, we put into $Q^{C_i}_n$ a copy of the structure $C^{a_n}_j$, where $j = \sigma_i(n)$.  As our isomorphism taking 
$C^{a_n}_j$ to the copy, we take the restriction of $p_n$ to the universe of $C^{a_n}_j$.  We complete the structure $C^a_i$ in the only way possible, letting 
$U^{C^a_i}_a$ be the union of the sets 
$U^{C^{a_n}_j}_{a_n}$, for $j = \sigma_i(n)$, and defining the ordering $<_a$ as prescribed.

It is clear that $\bbk_a$ is a computable representation of $\bk_a$.  Moreover, $E(\bbk_a)$ is computable.  We have 
$(i,j,f)\in E(\bbk_a)$ iff $dom(\sigma_i)\subseteq dom(\sigma_j)$, and for each $n\in dom(\sigma_i)$, if $\sigma_i(n) = i'$ and $\sigma_j(n) = j'$, and $f'$ is the finite partial function such that $f\circ p_n = p_n\circ f'$, then $(i',j',f')\in E(\bbk_{a_n})$.  From $\bbk_a$ and $E(\bbk_a)$, we obtain the Fra\"{i}ss\'{e} limit $D(\mathcal{A}_a)$ and $E(\bbk_a,\mathcal{A}_a)$.   

\begin{enumerate}

\item $Q_n^{\mathcal{A}_a}\restriction\tau_{a_n}\cong\mathcal{A}_{a_n}$,

\item $\phi_a$ characterizes the cardinal $\beths{|a|}$,

\item $(U_a,<_a)$ is a dense linear order without endpoints, and 

\item there is a model $\mathcal{B}$ of size $\beths{|a|}$ such that $(U_a^{\mathcal{B}},<_a^{\mathcal{B}})$ contains a 
co-final sequence of order type $\beths{|a|}$. 

\end{enumerate}
 
The last statement is true because, by the induction hypothesis, for each $n$, there exists a model $\mathcal{B}_{\alpha_n}$ such that $(U^{\mathcal{B}_{\alpha_n}}_{a_n},<_{a_n}^{\mathcal{B}_{\alpha_n}})$ contains a co-final sequence of order type $\beths{|a_n|}$.  Let $\mathcal{B}$ be the  $\tau_a$-structure that agrees with each $\mathcal{B}_{\alpha_n}$ on $Q_n$. Then $<^{\mathcal{B}}_a$ contains a co-final sequence of order type $\beth_{|a|}$.  It also follows from what we have said above that that $\tau_a$ is a computable relational vocabulary and $\mathcal{A}_a$ is a computable $\tau_a$-structure with Scott sentence $\phi$ characterizing $\beth_{|a|}$.  To complete the proof of Theorem \ref{ComputableBeths}, we observe the following.

\begin{cor}

We have a partial computable function $I$ that, for each $a\in \O$, gives indices for $\tau_a$, $\bbk_a$, $E(\bbk_a)$, 
$D(\mathcal{A}_a)$, and $E(\bbk_a,\mathcal{A}_a)$, where $\mathcal{A}_a$ has a Scott sentence $\phi_a$ characterizing 
$\beth_{|a|}$.  

\end{cor}

\begin{cor}
The Hanf number for Scott sentences of hyperarithmetical structures is $\beth_{\omega_1^{CK}}$.  
\end{cor}

We already remarked that $\beth_{\omega_1^{CK}}$ is an upper bound on the Hanf number.  The computable structures $\mathcal{A}_a$ from the previous section witness that it is also a lower bound.    

\bigskip
\noindent
\textbf{Remark}.  If $A$ is a countable admissible set with ordinal $\gamma$, then the Hanf number for Scott sentences for structures in $A$ is $\beth_\gamma$.  We already remarked that $\beth_\gamma$ is an upper bound.  Using essentially the same construction as in this section, we could determine a function $I$, $\Sigma$-definable in $A$, taking each ordinal $\alpha < \gamma$ to a tuple of elements of $A$, consisting of a vocabulary $\tau_\alpha$, a representation $\bbk_\alpha$ of an age $\bk_\alpha$, the embedding relation $E(\bbk_\alpha)$, the limit structure $\mathcal{A}_\alpha$ obtained effectively from $\bbk_\alpha$ and $E(\bbk_\alpha)$, and the relation $E(\bbk_\alpha,\mathcal{A}_\alpha)$.

\end{document}